\documentclass[a4paper]{article}
\usepackage[bottom=2in]{geometry} % full-width

\usepackage{amsmath}
\usepackage{amsthm}
\usepackage{amssymb}

\usepackage[utf8]{inputenc}
\usepackage{hyperref}
\hypersetup{
	unicode,
	pdfauthor={Author One, Author Two, Author Three},
	pdftitle={Uncertain standard quadratic optimization under distributional assumptions},
	pdfsubject={A simple article template},
	pdfkeywords={article, template, simple},
	pdfproducer={LaTeX},
	pdfcreator={pdflatex}
}

\usepackage[sort&compress,numbers,square]{natbib}
\theoremstyle{plain}
\newtheorem{theorem}{Theorem}
\newtheorem{corollary}[theorem]{Corollary}

\newtheorem{proposition}[theorem]{Proposition}

\theoremstyle{definition}
\newtheorem{definition}[theorem]{Definition}

\newtheorem{remark}[theorem]{Remark}

\newtheorem{property}[theorem]{Property}
\newtheorem{observation}[theorem]{Observation}

\usepackage{graphicx, color}
\graphicspath{{fig/}}

\usepackage{algorithm, algpseudocode} % use algorithm and algorithmicx for typesetting algorithms
\usepackage{mathrsfs} % for \mathscr command
\usepackage{bm}

\usepackage{bbm}    

\def\Ab{{\mathsf A}}

\def\Cb{{\mathsf C}}

\def\Gb{{\mathsf G}}

\def\Ib{{\mathsf I}}

\def\Lb{{\mathsf L}}
\def\Mb{{\mathsf M}}

\def\Qb{{\mathsf Q}}
\def\Rb{{\mathsf R}}
\def\Sb{{\mathsf S}}

\def\Ub{{\mathsf U}}

\def\Wb{{\mathsf W}}

\def\Yb{{\mathsf Y}}

\def\cc{{\mathbf c}}

\def\e{{\mathbf e}}

\def\r{{\mathbf r}}

\def\x{{\mathbf x}}
\def\y{{\mathbf y}}

\def\T{^\top}

\usepackage{enumitem}% roman

\usepackage{indentfirst} %The amount by which the first line of a paragraph is indented is controlled by a parameter command called \parindent. 

\usepackage{subfig}

\usepackage{multirow} % For multirows
\usepackage{amssymb} % For the checkmark
\usepackage{graphicx} % For the rotation
\usepackage{array} % To enlarge the vertical space of the table's contents

\title{Uncertain standard quadratic optimization\\ under distributional assumptions:\\
a chance-constrained epigraphic approach}
\author{Immanuel M. Bomze \thanks{VGSCO, Research Network Data Science and Faculty of Mathematics, University of Vienna, Oskar-Morgenstern-Platz 1, A-1090 Vienna, Austria. E-mail: immanuel.bomze@univie.ac.at} \and Daniel de Vicente\thanks{VGSCO and University of Vienna, Oskar-Morgenstern-Platz 1, A-1090 Vienna, Austria. Corresponding author. E-mail: a12032143@unet.univie.ac.at} }

\date{April 8, 2025}

\begin{document}
\maketitle

\begin{abstract} 

 \noindent The standard quadratic optimization problem (StQP) consists of minimizing a quadratic form over the standard simplex. Without convexity or concavity of the quadratic form, the StQP is NP-hard. This problem has many relevant real-life applications ranging from portfolio optimization to pairwise clustering and replicator dynamics. 

\noindent Sometimes, the data matrix is uncertain. We investigate models where the distribution of the data matrix is known but where both the StQP after realization of the data matrix and the here-and-now problem are indefinite. We test the performance of a chance-constrained epigraphic StQP to the uncertain StQP.
\end{abstract}

% ------------------------------------------
%   1. Introduction 
% ------------------------------------------
\section{Introduction}\label{sec:intro}

The standard quadratic optimization problem (StQP) consists of minimizing a quadratic form over the standard simplex
\begin{equation*}
%\tag{StQP}
\ell(\Qb):=\min_{\x \in \Delta} \,\x^{\top}\Qb\x {\color{black}\, ,}
\end{equation*}
where $\Qb \in \mathbb{R}^{n\times n}$ is a symmetric matrix, and $\Delta:= \{\x \in \mathbb{R}^n: \e^{\top}\x = 1, \x \geq \boldsymbol{0}\}$ is the standard simplex in $\mathbb{R}^n$. Here $\e\in \mathbb{R}^n$ is the vector of all ones and ${}^{\top}$ denotes transposition; $\Ib_n:= {\rm Diag} (\e)$ denotes the $n\times n$ identity matrix.  The objective function is already in general form since any general quadratic objective function $\x^{\top}\Ab\x + 2\cc^{\top}\x$ can be written in homogeneous form by defining the symmetric matrix $\Qb:= \Ab + \cc\e^{\top} + \e\cc^{\top}$.

Even though the StQP is simple - minimization of a quadratic function under linear constraints - it is NP-hard without assumptions on the definiteness of the matrix $\Qb${\color{black}. O}bserve that convex, but also concave instances are polynomially solvable, the latter even in closed form: $\ell(\Qb) = \min_{i} \Qb_{ii}$. Note that $\ell(\Qb)\ge 0$ is possible even if $\Qb$ is not positive semi-definite. In fact, the condition $\ell(\Qb)\ge 0$ characterizes copositivity~\cite{ShakedMondererBermanBook2021} of $\Qb$, and follows if no entry of $\Qb$ is negative (as, e.g. in the instances $\Qb_i^{\rm (nom)}$  generated in Section~\ref{sec:numerical_experiments} below). 

Irrespective of the sign of $\ell(\Qb)$, its calculation can be hard for indefinite instances: indeed, Motzkin and Straus \cite{motzkin-straus-1965} showed that the maximum clique problem, a well-known NP-hard problem, can be formulated as an StQP. Hence, the StQP is often regarded as the simplest of hard problems \cite{bomze2018complexity} since it contains the simplest non-convex objective function which is a quadratic form, and the simplest polytope as feasible set. Still, the StQP is a very flexible optimization class that allows for modelling of diverse problems such as portfolio optimization problems  \cite{markowitz-1952}, pairwise clustering \cite{pavan2003dominant} and replicator dynamics \cite{bomze1998standard}. Despite of its continuous optimization nature, it also serves to model discrete problems like the maximum-clique problem as well, as witnessed by above references.

The only data required to fully characterize an StQP is the data matrix $\Qb$. However, in many applications the matrix $\Qb$ is uncertain. StQPs with uncertain data have been explored in the literature. One of the most natural ways to deal with uncertain objective functions {\color{black}is via robust optimization \cite{ben2009robust}. In that paradigm, the decision-maker has to decide upon an uncertainty set $\mathcal U$ which encapsulates all the known information about the uncertain parameter. The uncertain parameter (in this case the uncertain data matrix $\Qb$) is suppossed to reside within the uncertainty set $\mathcal U$, as violations are not allowed \cite{gorissen2015practical}}.  Bomze et al.~\cite{Bomz20a} introduced the concept of a robust standard quadratic optimization problem, which they formulated as a minimax problem
\begin{equation}\label{robust_StQP}
\min_{\x \in \Delta} \max_{\Ub \in \mathcal U} \,\x^{\top}(\Qb^{\rm (nom)} + \Ub)\x
\end{equation}
with uncertainty set $\mathcal U$. The uncertain matrix $\Qb$ consisted of a nominal part $\Qb^{\rm (nom)}$ and an uncertain additive perturbation $\Ub$. In their paper, the authors investigated various uncertainty sets and proved that the copositive relaxation gap is equal to the minimax gap. Moreover, they observed that the robust StQP \eqref{robust_StQP} reduces to a deterministic StQP for many frequently used types of uncertainty sets $\mathcal U$. 

Passing from a robust to stochastic setting with known expectation, a natural alternative to get rid of the uncertainty is the here-and-now problem  (random quantities are designated by a tilde sign)
\begin{equation}\label{here-and-now_StQP}
\min_{\x \in \Delta} \, \mathbb E [\x^{\top}\widetilde \Qb\x] = \min_{\x \in \Delta} \,\x^{\top}\mathbb E [\widetilde \Qb]\x
\end{equation}
where the uncertain matrix $\widetilde \Qb$ is replaced by its expectation $\mathbb E [\widetilde \Qb]$. 

Bomze et al. \cite{bomze2022two} investigated a two-stage setting where the principal submatrix was deterministic and the rest of the entries followed a known probability distribution. {\color{black}In this paper, we propose an alternative to the here-and-now problem. As opposed to \cite{bomze2022two} we will} assume that the full data matrix $\widetilde \Qb$ is random according to a known distribution $\mathbb P$.
%, in which case models with two stages are less meaningful. 
The purpose of this note is to introduce, apparently for the first time, chance constraints for this problem class by introduction of an epigraphic variable, and moreover, to present a deterministic equivalent StQP formulation under reasonable distributional assumptions. {\color{black}Furthermore}, we establish a close connection of our new model to robustness with Frobenius ball uncertainty sets.

% ------------------------------------------
%   2. Chance-constrained epigraphic StQP 
% ------------------------------------------
\section{Chance-constrained epigraphic models of random StQPs}\label{sec:StQP_cce}

\subsection{
Definition and Value-at-Risk}

\begin{definition}[Chance-Constrained Epigraphic Standard Quadratic Optimization Problem (CCEStQP)]
Let $\widetilde \Qb$ be a random symmetric matrix with known distribution $\mathbb P$ and let $\alpha \in (0,1)$ be a given confidence level. Then the CCEStQP is defined by the problem
\begin{equation}\label{StQP_cce}\tag{CCEStQP}
%\tag{CCEStQP}
\begin{array}{lcl}
\ell^{\rm (cce)}_{\mathbb P, \alpha} :=&  \underset{\x,t}{\min} & t\\
&\text{s.t.} & \mathbb P[\x^{\top}\widetilde \Qb\x \leq t] \geq \alpha\\
&& \x \in \Delta\, ,\, t\in \mathbb{R}\, .
\end{array}
\end{equation}
\end{definition}

The reason for the name of CCEStQP can be explained as follows. It is well known that any optimization problem 
\begin{equation*}
\min_{\x \in \mathcal{X}} \, f(\x)
\end{equation*}
can be rewritten equivalently in epigraphic form by introducing an extra variable $t$
\begin{subequations}
\begin{align}
\underset{\x,t}{\min} & \quad t \nonumber\\
\text{s.t.} & \quad f(\x) \leq t \label{soft_constraint}\\
& \quad \x \in \mathcal{X}, \, t\in \mathbb{R}. \label{hard_constraint}
\end{align}
\end{subequations}

\noindent Now suppose that $f:\mathbb R^n \times \mathbb R^m \to \mathbb R$ is a random objective function that depends on the decision variable $\x \in \mathcal X \subseteq \mathbb R^n$ and on the random vector $\boldsymbol{\tilde \xi} \in \Xi \subseteq \mathbb R^m$. For every fixed $\x \in \mathcal X$, $f(\x,\boldsymbol{\tilde \xi})$ is a random variable with a distribution $\mathbb P$. Then we can regard constraint \eqref{soft_constraint} as a soft constraint and request feasibility with probability of at least $\alpha$. Constraint \eqref{hard_constraint} remains a hard constraint and must be satisfied for all feasible points. This has a familiar interpretation familiar in risk management as the so-called Value-at-Risk (VaR) of $f(\x,\boldsymbol{\tilde \xi})$ at $\x$ with confidence level $\alpha \in (0,1)$,  defined as~\cite{hull2012risk}  
\begin{equation*}
\text{VaR}_{\alpha}(\x):=\min \, \{ t: \mathbb P[f(\x,\boldsymbol{\tilde \xi}) \leq t] \geq \alpha\}.
\end{equation*}
The problem of minimizing VaR~\cite{Larsen2002} is %defined as 
\begin{equation}\label{prob_min_var}
\begin{array}{lcl}
\underset{\x \in \mathcal X}{\min}\,\text{VaR}_{\alpha}(\x) =& \underset{\x,t}{\min} & t\\
&\text{s.t.} & \mathbb P[f(\x,\boldsymbol{\tilde \xi}) \leq t] \geq \alpha\\
&& \x \in \mathcal X, \, t\in \mathbb{R}\, .
\end{array}
\end{equation}

\iffalse
The following observation is helpful and immediate by minimization in the epigraphic variable.
\begin{observation}
    Assume that the distributions of the random variables $f(\x,\boldsymbol{\tilde \xi})$  %$\widetilde\Qb$ be such that the random variable $\x^{\top}\widetilde \Qb\x$ 
    are continuous for all $\x\in \mathcal X$. Then problem~\eqref{prob_min_var}
    %the chance-constrained epigraphic StQP \eqref{StQP_cce} 
    is equivalent to the following chance-constrained optimization problem
$$
\begin{array}{cl}
\underset{\x,t}{\min} & t\\
\text{s.t.} & \mathbb P(f(\x,\boldsymbol{\tilde \xi}) \leq t) = \alpha\\
& \x \in \mathcal X\, ,\,t\in \mathbb{R}\, ,
\end{array}
$$
so that the inequality constraint has been replaced with an equality constraint. 
\end{observation}
\fi

El Ghaoui et al.~\cite{ghaoui2003worst} investigated problem~\eqref{prob_min_var} in the context of worst-case VaR for portfolio optimization with a return function
$$
f(\x,\boldsymbol{\tilde \xi}) := \x^{\top}\boldsymbol{\tilde \xi}
$$
where $\boldsymbol{\tilde \xi}$ was a vector of returns. 
Inspired by the VaR and passing from linear to quadratic expressions in the decision variable $\x$, we want to find the smallest number $t \in \mathbb R$ such that $\x^{\top}\widetilde \Qb\x \leq t$ with probability at least $\alpha$, where $\alpha \in (0,1)$ is a confidence level provided by the decision-maker. 

\subsection{Distributional models for random (indefinite) symmetric matrices}

We are interested in instances of \eqref{StQP_cce} where the random matrix $\widetilde \Qb$ is indefinite. Two common and simple ways of generating  an indefinite random symmetric matrix $\widetilde \Qb$ are:
\begin{enumerate}[label=(\roman*)]
    \item Generate a random $n\times n$ matrix $\widetilde \Rb$ and put $\widetilde\Qb :=\Qb^{\rm (nom)} + \frac{\beta}{\sqrt{2}}(\widetilde \Rb+\widetilde\Rb^{\top})$, for some nominal matrix $\Qb^{\rm (nom)}$ and a suitable parameter $\beta$.
    \item Generate a random $n\times p$ matrix $\widetilde \Yb$ and put $\widetilde\Qb :=\widetilde\Yb\widetilde\Yb^{\top} - \eta\,\Ib_n$. The parameter $\eta >0$ has to be large enough such that the positive semi-definite matrix $\widetilde\Yb\widetilde\Yb^{\top}$ becomes indefinite after subtraction of $\eta\,\Ib_n$, but not too large to avoid negative-definite matrices $\widetilde \Qb$.
\end{enumerate}

In case that $\widetilde \Rb$ or $\widetilde \Yb$ are normal with independent and identically distributed (i.i.d.) entries we obtain the shifted and scaled Gaussian Orthogonal Ensemble in the first case and the shifted Wishart Ensemble in the second case~\cite{geman1980limit,marchenko1967distribution,silverstein1985smallest}. {\color{black} We formalize this observation in the following definitions.

\begin{definition}
\begin{enumerate}[label=(\roman*)] 
\item Let $\widetilde \Gb = (\widetilde g_{ij})$ be a random symmetric matrix such  that $\widetilde g_{ii}$ are i.i.d. with $\widetilde g_{ii} \sim \mathcal{N}(0,2)$, for $i=1,\dots,n$ and $\widetilde g_{ij}$ are i.i.d. standard normal random variables, $\widetilde g_{ij} \sim \mathcal{N}(0,1)$, for $1\leq i < j \leq n$, then $\widetilde \Gb$ is called a Gaussian Orthogonal Ensemble (GOE) matrix.
\item Let $\Qb^{\rm (nom)} \in \mathbb R^{n\times n}$ be a nominal matrix, $\beta \neq 0$ and $\widetilde \Gb$ be a GOE matrix, then the GOE perturbation model is defined as 
\begin{equation}\label{def_Q_normal}
    \widetilde \Qb = \Qb^{\rm (nom)} + \beta \, \widetilde \Gb\,.
\end{equation}
\end{enumerate}
\end{definition}

\begin{definition}
\begin{enumerate}[label=(\roman*)] 
\item Let $\boldsymbol{\Sigma}\in \mathbb R^{n\times n}$ be a positive-definite matrix and \sloppy$\widetilde\y_1,\dots,\widetilde\y_p\sim \mathcal N_n(\boldsymbol{0},\boldsymbol{\Sigma})$ be i.i.d. random vectors from the $n$-variate normal distribution with mean $\boldsymbol{0}$ and covariance matrix $\boldsymbol{\Sigma}$. Set $\widetilde \Yb = (\widetilde\y_1,\dots,\widetilde\y_p)$, then the matrix $\widetilde \Wb = \widetilde \Yb \widetilde \Yb^{\top}$ is said to have a Wishart distribution with covariance matrix $\boldsymbol{\Sigma}$ and $p$ degrees of freedom.

\item Let $\widetilde \Wb$ be a Wishart matrix  and $\eta > 0$, then the shifted Wishart model is defined as 
\begin{equation}\label{def_Q_wishart}
    \widetilde \Qb = \widetilde \Wb - \eta \, \Ib_n\,.
\end{equation}
\end{enumerate}
\end{definition}

For these distributions, we derive the following location-scale property of the distributions of the quadratic forms $\x\T\widetilde \Qb\x$.} 

\begin{proposition}\label{propgoe}
{\color{black} For the GOE perturbation model \eqref{def_Q_normal} and the shifted Wishart model \eqref{def_Q_wishart} there are two symmetric $n\times n$ matrices $\Mb$ and $\Sb$, the latter satisfying $\x\T\Sb\x>0$ for all $\x\in \Delta$, and a continuous cumulative distribution function (cdf) $F$ which is strictly increasing on $\{ t\in \mathbb R : F(t)>0\}$, such that
$$
\mathbb P[\x^{\top}\widetilde \Qb \x \leq t] = F\left(\frac{t - \mu(\x)}{\sigma(\x)}\right)\,,
$$
where 
$$
\mu(\x) := \x^{\top}\Mb\x \quad  \text{  and  } \quad \sigma(\x) := \x^{\top}\Sb\x > 0
$$
are location and scale parameters, respectively.}\end{proposition}
\begin{proof}
{\color{black} We first deal with the 
GOE perturbation model \eqref{def_Q_normal}. Obviously, } for any decision vector $\x\in \mathbb R^n$, the quadratic form $\x^{\top}\widetilde \Gb \x$ is a normally distributed random variable with mean 0 and variance $2||\x||_2^4$, i.e. 
$$
\x^{\top}\widetilde \Gb \x \sim \mathcal{N}(0,2||\x||_2^4)\, .
$$
{\color{black}Now putting
$\Mb = \Qb^{\rm (nom)}$, $\Sb = \beta\,\Ib_n$ with $\beta >0$ and choosing $F$ as the
%$$
%F(t) = \Phi(t) = \frac{1}{\sqrt{2\pi}}\int_{-\infty}^te^{-u^2/2}du
%$$
as the cdf of the standard normal distribution (denoted by $\Phi$), we see that also in this case the assertion holds true.\\
Now we turn to the shifted Wishart model \eqref{def_Q_wishart}. } We have that $\x^{\top}\widetilde\Wb\x$ is a gamma distributed random variable with shape $k= p/2$ and scale $\theta= 2\x^{\top}\boldsymbol{\Sigma}\x$, where $p$ determines the shape of $F$. {\color{black} The matrices are chosen as $\Mb = -\eta\,\Ib_n$ and $\Sb = 2\boldsymbol{\Sigma}$.}
\end{proof}

{\color{black} Before we proceed with our results, let us highlight the real-world relevance of the GOE perturbation model which has been extensively explored in the literature, see \cite{aizenman2017matrix, shapiro2016level}.  A GOE perturbation model \eqref{def_Q_normal}} has several applications, among them portfolio optimization, where the portfolio consists of $n$ assets with historical mean returns $\r$ and historical covariance matrix $\Cb$. {\color{black}{In that setting, the vector $\x = (x_1,\dots,x_n)^{\top}$ represents the allocation of the total budget (normalized to 1) to the assets. In particular, the entry $x_i$ represents the portion of the budget devoted to the $i$-th asset. The goal is to find an allocation vector $\x$} which} minimizes the expected risk and maximizes the expected return
\begin{equation}\label{mean-variance}
\min_{\x \in \Delta} \,\x^{\top}\Cb\x - \r^{\top}\x\, .
\end{equation}
Note that in this model, short-selling is not allowed. By definition of the feasible set $\Delta$, problem \eqref{mean-variance} is equivalent to the StQP
\begin{equation*}
\min_{\x \in \Delta} \,\x^{\top}\Qb^{\rm (nom)}\x\, ,
\end{equation*}
where 
$$
\Qb^{\rm (nom)}:= \tfrac{1}{2}(\Cb- \r \e^{\top} - \e \r^{\top})\,.
$$  
Suppose that in the estimation of the covariance matrix $\Cb$ and expected return vector $\r$ measurement errors were made so that we replace the matrix $\Qb^{\rm (nom)}$ by an additive perturbation with a GOE matrix $\widetilde \Gb$ and model the mean-variance portfolio optimization problem as an uncertain StQP
\begin{equation*}
\min_{\x \in \Delta} \,\x^{\top}(\Qb^{\rm (nom)} + \beta \, \widetilde \Gb)\x\, ,
\end{equation*}
where $\beta$ is a parameter denoting the amplitude of the perturbation.
%\end{remark}

It is clear that under assumption \eqref{def_Q_normal} the here-and-now problem \eqref{here-and-now_StQP} consists of solving the following problem
\begin{equation}\label{StQP_GOE_han}
\min_{\x \in \Delta} \, \x^{\top}\Qb^{\rm (nom)}\x.   
\end{equation}
Since no assumptions on the definiteness of $\Qb^{\rm (nom)}$ are made, the here-and-now problem \eqref{StQP_GOE_han} can be potentially indefinite.

\section{A distributional assumption and main results}

\subsection{A reasonable distributional assumption on random StQPs}
{\color{black} As motivated by above observations, it is reasonable} to suppose that the general distribution $\mathbb P$ of $\widetilde \Qb$ enjoys the following {\color{black}location/scale type property for the quadratic forms.}

\begin{property}\label{property}
{\color{black} Consider two symmetric $n\times n$ matrices $\Mb$ and $\Sb$, the latter satisfying $\x\T\Sb\x>0$ for all $\x\in \Delta$, and a continuous cdf $F$ which is strictly increasing on $\{ t\in \mathbb R : F(t)>0\}$.}\\
 \sloppy For all  $\x\in \Delta$, the distribution of $\x^{\top}\widetilde \Qb \x$ is of location/scale type:
$$
\mathbb P[\x^{\top}\widetilde \Qb \x \leq t] = F\left(\frac{t - \mu(\x)}{\sigma(\x)}\right)\,,
$$
with {\color{black} location parameter $\mu$ and scale parameter $\sigma$} depending on $\x$ in the following way:
$$
\mu(\x) := \x^{\top}\Mb\x, \quad  \sigma(\x) := \x^{\top}\Sb\x > 0\,.
$$
\end{property}

\begin{remark} \textit{\color{black} Not every location/scale family of distributions of quadratic forms satisfies Property~\ref{property}, but as observed in Proposition~\ref{propgoe}, the most common distributional models indeed enjoy Property~\ref{property}.}
\end{remark}

\subsection{Deterministic reformulation}

\noindent {\color{black} Now we show that Property~\ref{property} ensures that CCEStQP can be reformulated as a deterministic StQP.}
\begin{theorem} \label{mainthm} {\color{black} Assume that $\mathbb P$ fulfills Property~\ref{property} for some $(F, \Mb,\Sb)$  as outlined there. Then} the chance-constrained epigraphic StQP~\eqref{StQP_cce} is equivalent to a deterministic StQP. In particular,
$$
\ell^{\rm (cce)}_{\mathbb P, \alpha} = \min_{\x \in \Delta} \, \x^{\top}\overline{\Qb}\x \,,
$$
with $\overline \Qb:=\Mb + F^{-1}(\alpha)\Sb$.
\end{theorem}
\begin{proof}
A feasible pair $(\x,t) \in \Delta \times \mathbb R$ satisfies
$$
\mathbb P[\x^{\top}\widetilde \Qb \x \leq t] \geq \alpha \quad \text{ if and only if } \quad
\x^{\top}\overline \Qb \x \leq t.
$$    
Therefore $\eqref{StQP_cce}$ is equivalent to 
\begin{equation*}
\begin{array}{cl}
\underset{\x,t}{\min} & t\\
\text{s.t.} & \mathbb \x^{\top}\overline \Qb\x \leq t\\
& \x \in \Delta\, ,\, t\in \mathbb{R}\,,
\end{array}
\end{equation*}
which is precisely the epigraphic reformulation of the deterministic StQP 
$$
\min_{\x \in \Delta}\, \x^{\top}\overline{\Qb}\x\,.
$$
\end{proof}

Hence we have established a deterministic StQP equivalent to the CCEStQP under reasonable distributional assumptions. This is in line with several robust counterparts~\cite{Bomz20a}. 

\subsection{An important special case: GOE perturbation model}

{\color{black} Let us recapitulate the above result explicitly for the GOE perturbation model:}

\begin{corollary}\label{corgoe}
Let $\widetilde \Qb$ be defined as in \eqref{def_Q_normal}, then
$$
    \x^{\top}\widetilde \Qb\x \sim \mathcal N \left(\mu(\x), \sigma^2(\x)\right)\quad\mbox{with }
$$
$$
 \mu(\x)= \x^{\top}\Qb^{\rm (nom)}\x\mbox{ and }\sigma(\x) = \sqrt{2}\beta\, \x^{\top}\Ib_n\x \, .
$$
Therefore for the GOE perturbation model, the chance-constrained epigraphic problem \eqref{StQP_cce} is equivalent to a deterministic StQP
\begin{equation}\label{StQP_GOE_CCE}
\ell_\alpha^{(cce)}:= \min_{\x \in \Delta} \, \x^{\top}\Qb^{\rm (cce)}\x\,,
\end{equation}
with
$$
\Qb^{\rm (cce)}:= \Qb^{\rm (nom)} + \sqrt{2}\beta\, \Phi^{-1}(\alpha)\,\Ib_n
$$ 
where $\Phi$ denotes the cumulative distribution function of a standard normal random variable. % and .
\end{corollary}
\begin{proof} Straightforward from Proposition~\ref{propgoe} and Theorem~\ref{mainthm}.
\end{proof}

We proceed with the announced relation of above result to robust counterparts when uncertainty sets are Frobenius balls {\color{black} under the GOE perturbation model:}

\begin{theorem}\label{thm:ccestqp-and-rstqp}
Under the GOE perturbation model \eqref{def_Q_normal}, define
$$
\rho := \sqrt{2}\beta\, \Phi^{-1}(\alpha)\, .
$$
Then the CCEStQP~\eqref{StQP_GOE_CCE} can be interpreted as the deterministic counterpart of the robust StQP
$$
\min_{\x \in \Delta} \, \max_{\Ub \in \mathcal U} \,\x^{\top}( \Qb^{\rm (nom)}  + \Ub)\x\,, $$
with Frobenius ball uncertainty set
$$
\mathcal U = \{\Ub = \Ub^{\top}: || \Ub ||_F \leq \rho\}\, .
$$
\end{theorem}

\begin{proof}
Recall that the Frobenius norm is defined by $ ||\Ub||_F^2 := \operatorname{trace}(\Ub^{\top}\Ub)$. We use~\cite[Theorem 4 and Remark 5]{Bomz20a}, defining $\Lb:=-\rho\,(\Cb\Cb^{\top})^{-1}$ with $\Cb=\Ib_n$. Flipping the sense of optimization, we obtain the deterministic counterpart of the robust problem with uncertainty set $\mathcal U$
$$
\min_{\x \in \Delta} \, \max_{\Ub \in \mathcal U} \,\x^{\top}(\Qb^{\rm (nom)} + \Ub)\x = \min_{\x \in \Delta} \, \x^{\top}( \Qb^{\rm (nom)} +\rho\,\Ib_n)\x\, ,
$$
which is exactly~\eqref{StQP_GOE_CCE}. {\color{black} Note that $\rho\,\Ib_n$ is not an element of $\mathcal U$ for $n>1$ since $||\,\Ib _n\,||_F^2 = n$.}
\end{proof}

Observe that a similar result for the shifted Wishart model seems not to hold because of the negative shift in
$\Mb  = -\eta\, \Ib_n$. We expect that the more general distributional assumption of Property~\ref{property} would neither allow for such a characterization.

{\color{black} This is not the first time that a chance-constrained optimization problem is reformulated as a deterministic counterpart of a robust optimization problem, see  \cite[Chapters 2 and 4]{ben2009robust} and \cite[Section 3]{bertsimas2011theory}. However, \cite{ben2009robust,bertsimas2011theory} discuss models with chance constraints both linear in the uncertain parameter and linear in the decision variable. By contrast, in our case, even though the function $f(\x,\widetilde \Qb) = \x^{\top}\widetilde \Qb \x$ is linear in the random parameter $\widetilde\Qb$, it is nonlinear in the decision variable $\x$. Thus, Theorem \ref{thm:ccestqp-and-rstqp} distinguishes itself from what can be found in the literature.}

We close this section discussing a possible convexifying effect by passing from the indefinite nominal StQP to the 
CCEStQP.
Let $\lambda_{\rm max}$ and $\lambda_{\rm min}$ denote the largest and smallest eigenvalues of the nominal matrix $\Qb^{\rm (nom)}$, respectively. We want to study the cases where the here-and-now problem \eqref{StQP_GOE_han} is indefinite while the chance-constrained epigraphic problem \eqref{StQP_GOE_CCE} is convex. 

\begin{proposition}\label{prop:cce}
    Let $\alpha > 1/2$, $\beta > 0$, $\lambda_{\rm max} > 0 > \lambda_{\rm min}$ and $\Phi$ denote again the cumulative distribution function of the standard normal distribution. Then
\begin{equation*}
\Qb^{\rm (cce)}\mbox{ is positive semi-definite } \iff \alpha \geq \Phi\left(\frac{|\lambda_{\rm min}|}{\sqrt{2}\beta}\right) \, .
\end{equation*}
%Note that $\Qb^{\rm (cce)}$ will never be concave but may be indefinite, namely if $\alpha < \Phi\left(\frac{|\lambda_{\rm min}|}{\sqrt{2}\beta}\right)$.
\end{proposition}

\section{Numerical experiments}\label{sec:numerical_experiments}
\noindent %In this section 
{\color{black}We carried out %a number of 
experiments for the model discussed in Corollary~\ref{corgoe}.} %Section \ref{sec:GOE}. 
All results were computed using {\tt Gurobi} {\color{black}v.11.0.2.} Non-convex StQP instances were {\color{black} rewritten in bilinear form and then solved by spatial branching} with 60 seconds maximum runtime and gap tolerance $10^{-6}$. 

\subsection{Nominal instance generation}
\noindent We first set $n=30$ and generated 10 i.i.d.\ symmetric nominal matrices $\Qb_1^{(\rm nom)},\dots,\Qb_{10}^{(\rm nom)}$ component-wise from the uniform distribution on [0,1]
$$
(\Qb_i^{(\rm nom)})_{k\ell} \sim {\color{black}U}_{[0,1]} \text{ for } i = 1,\dots,10 \text{ and } 1 \leq k \leq \ell \leq n,
$$
$$
(\Qb_i^{(\rm nom)})_{k\ell} = (\Qb_i^{(\rm nom)})_{\ell k} \mbox{ if } k>\ell\, .
$$
%As to be expected, a
All generated matrices $\Qb_i^{(\rm nom)}$ were indefinite with $\lambda_{\rm min}(\Qb_i^{(\rm nom)}) \in [-3.2,-2.7]$ and \\
$\lambda_{\rm max}(\Qb_i^{(\rm nom)}) \in [14.4,15.8]$. 
\subsection{Random data generation}
We chose $\beta = 3$ and generated 100 i.i.d.\ GOE matrices $\Gb_1,\dots,\Gb_{100}$. All realizations 
$$
\Qb_{ij}:=\Qb_i^{(\rm nom)} + \beta\,\Gb_j, \quad i=1,\dots,10, \, j = 1,\dots,100\, ,
$$
were indefinite as well, 
%a consequence of Wigner's semicircle law (observe that %in this case the matrix  $\Qb_i^{(\rm nom)} - \frac{1}{2}\e\e^{\top}$ is a Wigner matrix) 
with
$
\lambda_{\rm min}(\Qb_{ij}) \in [-35.9,-25.9]$ and 
$\lambda_{\rm max}(\Qb_{ij}) \in [25.7,41.2]$. We considered confidence levels $\alpha \in \{0.55, 0.56,\dots, 0.99\}$ and defined
$$
\Qb^{\rm (cce)}_{i,\alpha}:= \Qb^{\rm (nom)}_i +  \sqrt{2}\beta\, \Phi^{-1}(\alpha)\,\Ib_n\, ,
$$ 
according to \eqref{StQP_GOE_CCE}. By Proposition \ref{prop:cce} since 
$$
\Phi\left(\frac{|\lambda_{\rm min}(\Qb^{\rm (nom)}_i)|}{\sqrt{2}\beta}\right) \leq \Phi\left(\frac{3.2}{3\sqrt{2}}\right) \approx 0.77\,,
$$
the $\Qb^{\rm (cce)}_{i,\alpha}$ were positive-definite approximately for $\alpha \geq 0.77$ and indefinite otherwise.

\subsection{Comparing solutions to the nominal with solutions to random instances}
We then solved problem \eqref{StQP_GOE_CCE} with for all $i$ and $\alpha$ and obtained solution pairs $(\x^{\rm (cce)}_{i,\alpha},t^{\rm (cce)}_{i,\alpha})$. 
We confirmed that the empirical probabilities are close to $\alpha$, {\color{black}i.e. 
$$
\frac{\#\Big\{j:\x^{\rm (cce)\top}_{i,\alpha}\Qb_{ij}\x^{\rm (cce)}_{i,\alpha}\leq t^{\rm (cce)}_{i,\alpha}\Big\}}{100} \approx \alpha\,,
$$
for all $i$ and $\alpha$, where $\#$ counts the number of elements in a set.} (In case $\alpha = 70\%$ they range in the interval $[67\%,77\% ]$). Then we solved all instances $\Qb_{ij}$ yielding
optimal values $\displaystyle\ell(\Qb_{ij})=\min _{\x\in \Delta} \x\T\Qb_{ij}\x$. For each instance $\Qb_{ij}$, either the global optimum or a local optimum within tolerance was found. %, as shown in Figure \ref{fig:global_local}.%sowieso we report on $|J_i|$ to control global optimality across all $i,j$.

\noindent For comparison, we report on the values $\x^{\rm (cce)\top}_{i,\alpha}\Qb^{\rm (nom)}_i\x^{\rm (cce)}_{i,\alpha}$ and $\ell(\Qb^{\rm (nom)}_i)$, and also compare the empirical values
$\x^{\rm (cce)\top}_{i,\alpha}\Qb_{ij}\x^{\rm (cce)}_{i,\alpha}$ and $\ell(\Qb_{ij})$, accumulated across all $j=1\dots,100$, and also summarized across all instances $i$, to avoid any instance selection bias. To be more precise, let 
$$
\ell^{\rm (nom)} := \frac{1}{10}\sum_{i=1}^{10}\ell(\mathsf Q_i^{\rm (nom)}),
$$
$$\ell^{\rm (emp)} := \frac{1}{1000}\sum_{i=1}^{10}\sum_{j=1}^{100}\ell(\mathsf Q_{ij}),
$$

$$
 \ell^{\rm (nom)}_{\rm cce,\alpha} := \frac{1}{10}\sum_{i=1}^{10}\mathbf x^{\rm (cce) \top}_{i,\alpha} \mathsf Q_i^{\rm (nom)}\mathbf x^{\rm (cce)}_{i,\alpha},
 $$
 $$
\ell^{\rm (emp)}_{\rm cce,\alpha} := \frac{1}{1000}\sum_{i=1}^{10}\sum_{j=1}^{100}\mathbf x^{\rm (cce) \top}_{i,\alpha} \mathsf Q_{ij}\mathbf x^{\rm (cce)}_{i,\alpha}\, .
$$

\begin{figure}[H]
    \centering
    \begin{minipage}{0.5\textwidth}
        \centering  
\includegraphics[width=0.95\textwidth]{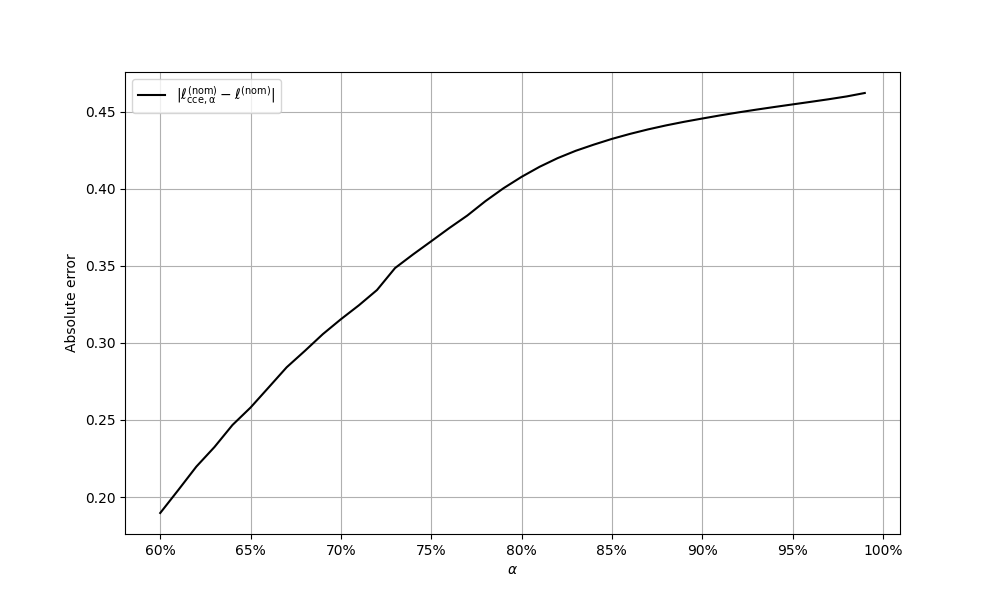}
    \caption{$|\ell^{\rm (nom)}_{\rm cce,\alpha} - \ell^{\rm (nom)}|$}
        \label{fig:1}
    \end{minipage}\hfill
    \begin{minipage}{0.5\textwidth}
        \centering   
\includegraphics[width=0.95\textwidth]{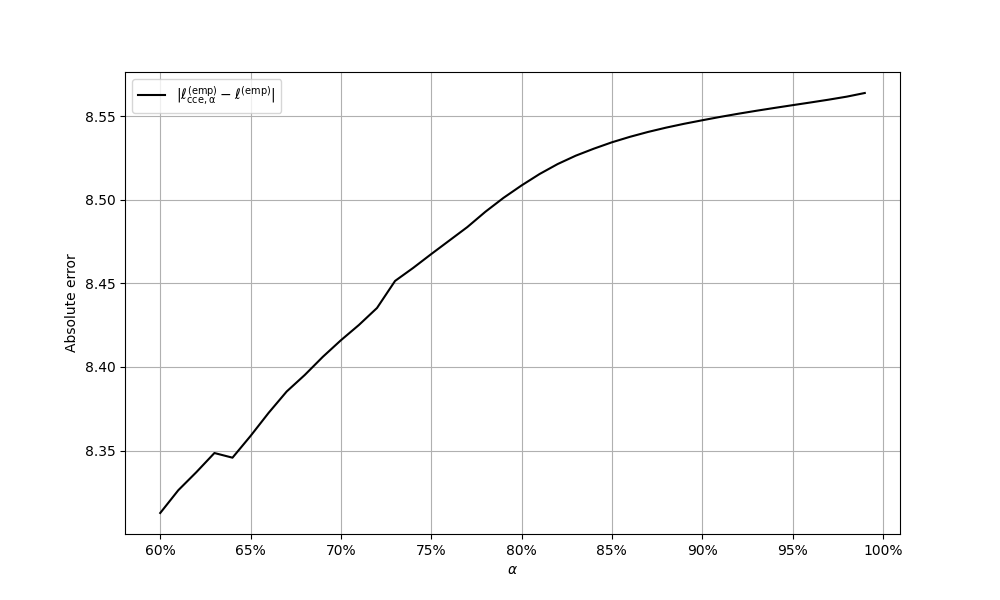}
\caption{$|\ell^{\rm (emp)}_{\rm cce,\alpha} - \ell^{\rm (emp)}|$}
        \label{fig:2}
    \end{minipage}
\end{figure}

{\color{black}{
\noindent Figure \ref{fig:1} shows the (average) absolute error of the solutions $\mathbf x^{\rm (cce)}_{i,\alpha}$ with respect to the optimal solutions of $\ell(\mathsf Q_i^{\rm (nom)})$ as a function of $\alpha$, i.e. the differences $|\ell^{\rm (nom)}_{\rm cce,\alpha} - \ell^{\rm (nom)} |$ while Figure \ref{fig:2} illustrates the (average) absolute error of the solutions $\mathbf x^{\rm (cce)}_{i,\alpha}$ with respect to the optimal solutions of $\ell\left(\mathsf Q_{ij}\right)$ as a function of $\alpha$, i.e. the differences $|\ell^{\rm (emp)}_{\rm cce,\alpha} - \ell^{\rm (emp)}|$. As shown in Figures \ref{fig:1} and \ref{fig:2}, the function values provided by the solutions $\mathbf x^{\rm (cce)}_{i,\alpha}$ are far away from the optimal values of $\displaystyle\min _{\x\in \Delta} \x\T\mathsf Q_i^{\rm (nom)}\x$ and $\displaystyle\min _{\x\in \Delta} \x\T\mathsf Q_{ij}\x$, respectively. However, $\mathbf x^{\rm (cce)}_{i,\alpha}$ are sometimes less conservative  compared to  robust methods, as showcased in the following section.}
%the differences are considerable, but sometimes less conservative  compared to  robust methods, as showcased in the following section.

\subsection{Comparing conservativity}
In the following experiment inspired by~\cite{Bomz20a}, 
} we constructed, for each scenario $i=1,\dots,10$, a robust StQP \eqref{robust_StQP} each with the choice of a box uncertainty set of the form
\begin{align}\label{uncertainty_set}
    \mathcal U_i = \left\{\Ub = \Ub^{\top} : \rho (\underline{\Qb_{i}} -\Qb^{\rm (nom)}_i) \leq  \Ub \leq  \rho (\overline{\Qb_{i} }-\Qb^{\rm (nom)}_i)\right\} \,, 
\end{align}

%\begin{equation}\label{uncertainty_set}
%\mathcal U_i = \left\{\Ub = \Ub^{\top} : \rho [\underline{\Qb_{i}} -\Qb^{\rm (nom)}_i]_{k\ell}\le  \Ub _{k\ell} \le  \rho [\overline{\Qb_{i} }-\Qb^{\rm (nom)}_i]_{k\ell}\mbox{ for all }k, \ell\right\}    
%\end{equation}
\noindent where 
$$
(\underline{\Qb_{i}})_{k\ell}  := \min_{j} \, (\Qb_{ij})_{k\ell}, \quad \quad (\overline{\Qb_{i}})_{k\ell}  := \max_{j} \, (\Qb_{ij})_{k\ell}
$$
%$\underline\Qb^{(i)} _{k,\ell}\le [\Qb_{ij}]_{k,\ell} \le  \overline\Qb^{(i)}_{k\ell}$ for all $j\in J_i$, 
defines a box covering all realizations (depending on the instance number $i$), and $\rho\in (0,1]$ is a parameter controlling the size of $\mathcal U_i$. By application of~\cite[Theorem~3]{Bomz20a} to the 
 constructed uncertainty sets $\mathcal U_i$, we obtain
\begin{align*} 
\min_{\x \in \Delta} &\, \max_{\Ub \in \mathcal U_i} \,\x^{\top}(\Qb_i^{\rm (nom)} + \Ub)\x\\
&= \min_{\x \in \Delta} \left [\x^{\top}\Qb_i^{\rm (nom)}\x + \max_{\Ub \in \mathcal U_i} \,\x^{\top}\Ub\x\right ] \\
&= \min_{\x \in \Delta} \,\x^{\top}\Qb_i^{\rm (nom)}\x + \rho\,\x^{\top}( \overline{\Qb_{i}} -\Qb_i^{\rm (nom)} )\x\\
&=  \min_{\x \in \Delta} \,\x^{\top}[(1-\rho)\,\Qb_i^{\rm (nom)} + \rho \,\overline{\Qb_{i}}]\x\, .
\end{align*}
{\color{black} If $\rho=1$, the uncertainty set $\mathcal U_i $ would cover all realizations (in our simulation case randomly generated); to avoid any bias against the robust model in case of outliers among the realizations which would incur overly conservative solutions, we decided to}
set $\rho = 0.8$ and denote by $\x_{i}^{\rm (rob)}$ the robust solution  obtained from the robust StQPs \eqref{robust_StQP} with uncertainty sets  \eqref{uncertainty_set}. Let 
$$
\ell^{\rm (nom)}_{\rm rob }  := \frac{1}{10}\sum_{i=1}^{10}\mathbf x_{i}^{\rm (rob)\top}\mathsf Q_i^{\rm (nom)}\mathbf x_{i}^{\rm (rob)}\,,
$$

$$
\ell^{\rm (emp)}_{\rm rob} := \frac{1}{1000}\sum_{i=1}^{10}\sum_{j=1}^{100}\mathbf x_{i}^{\rm (rob)\top}  \mathsf Q_{ij}\mathbf x_{i}^{\rm (rob)} \, . 
$$

{\color{black}
We compare the obtained realized feasible values 
%$\x^{\rm (rob)\top}_{i,\rho}\Qb_{ij}\x^{\rm (rob)}_{i,\rho}$ with $\ell(\Qb_{ij})$ and  $\x^{\rm (cce)\top}_{i,\alpha}\Qb_{ij}\x^{\rm (cce)\top}_{i,\alpha}$ 
in a similar way as above.   
%observing the differences $\ell^{\rm (nom)}_{\rm rob}-\ell^{\rm (nom)}$ and  $\ell^{\rm (emp)}_{\rm rob}-\ell^{\rm (emp)}$ with 
Figure \ref{fig:3} depicts the (average) absolute error of the solutions $\mathbf x^{\rm (cce)}_{i,\alpha}$ and  
$\mathbf x^{\rm (rob)}_{i}$ with respect to the optimal solutions of $\ell(\mathsf Q_i^{\rm (nom)})$, i.e. the differences $|\ell^{\rm (nom)}_{\rm cce,\alpha} - \ell^{\rm (nom)} |$ as a solid line and $|\ell^{\rm (nom)}_{\rm rob} - \ell^{\rm (nom)}|$ as a dashdotted line and 
Figure \ref{fig:4} presents the (average) absolute error of the solutions $\mathbf x^{\rm (cce)}_{i,\alpha}$ and  
$\mathbf x^{\rm (rob)}_{i}$ with respect to the optimal solutions of $\ell(\mathsf Q_{ij})$, i.e. the differences $|\ell^{\rm (emp)}_{\rm cce,\alpha} - \ell^{\rm (emp)} |$ as a solid line and $|\ell^{\rm (emp)}_{\rm rob} - \ell^{\rm (emp)}|$ as a dashdotted line.

\begin{figure}[H] 
    \centering
    \begin{minipage}{0.5\textwidth}
        \centering
    \includegraphics[width=0.95\textwidth]{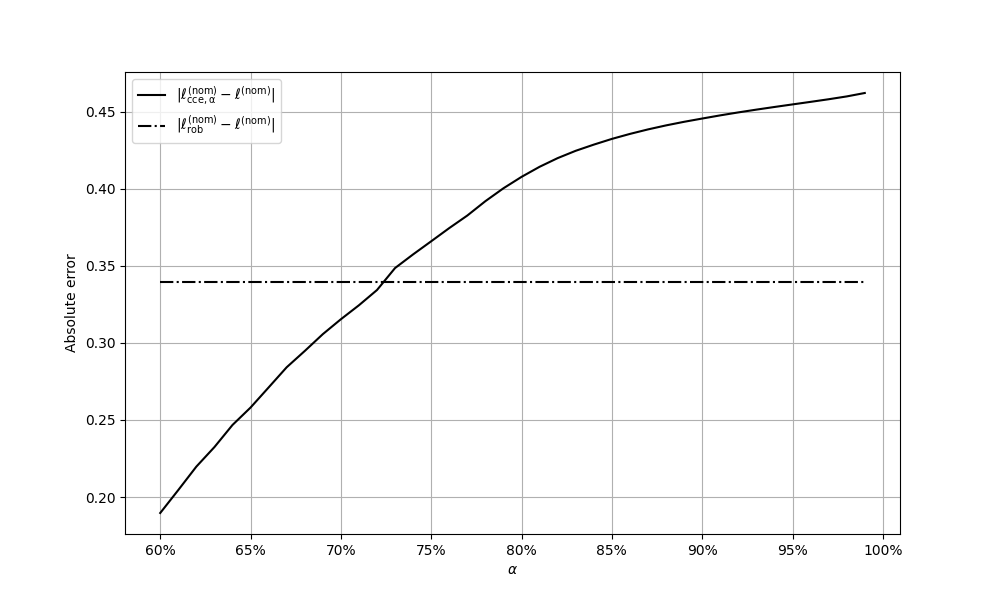}
        \caption{$|\ell^{\rm (nom)}_{\rm cce,\alpha} - \ell^{\rm (nom)}|$ vs $|\ell^{\rm (nom)}_{\rm rob } - \ell^{\rm (nom)}|$}
        \label{fig:3}
    \end{minipage}\hfill
    \begin{minipage}{0.5\textwidth}
        \centering
    \includegraphics[width=0.95\textwidth]{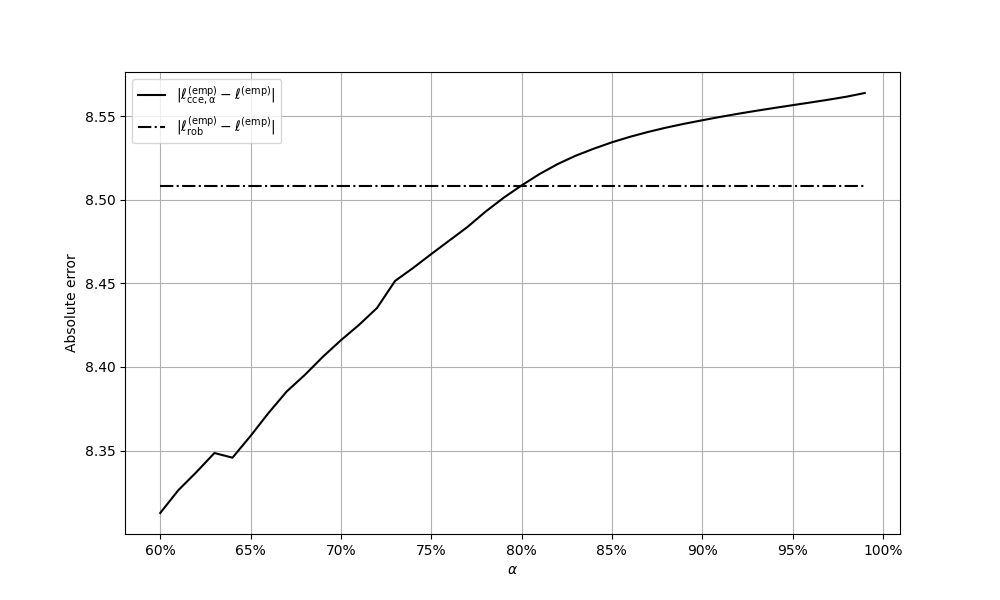}
        \caption{$|\ell^{\rm (emp)}_{\rm cce,\alpha}-\ell^{\rm (emp)}|$ vs $|\ell^{\rm (emp)}_{\rm rob}-\ell^{\rm (emp)}|$}
        \label{fig:4}
    \end{minipage}
\end{figure}

\noindent As one can see in Figure \ref{fig:3}, for the nominal problem $\ell(\Qb_i^{\rm (nom)})$ the chance-constrained epigraphic solutions $\mathbf x^{\rm (cce)}_{i,\alpha}$ are less conservative than the robust solutions $\mathbf x^{\rm (rob)}_{i}$ for all $\alpha \leq 0.72$. For the empirical problem $\ell(\Qb_{ij})$, Figure \ref{fig:4} shows that any $\alpha \leq 0.79$ yields less conservative chance-constrained epigraphic solutions than the robust solutions. }

\section{Conclusion}
\noindent We introduced and motivated the chance-constrained epigraphic StQP, a new model for solving uncertain StQPs under distributional assumptions, and established a deterministic counterpart as an instance in the same problem class (another StQP). Our findings parallel similar observations on robust StQPs, and indeed a special variant of this model is equivalent  to a particular robust formulation. However, preliminary experiments seem to suggest that the chance-constrained epigraphic StQP can be less conservative than a robust approach, if the confidence level (probabilistic optimality guarantee) is not too large.

\section*{Acknowledgements}
\noindent The authors thank Christa Cuchiero and Abdel Lisser for stimulating discussions and valuable suggestions. Both authors are indebted to VGSCO for financial support enabling presentation of this work at various major conferences: EurOpt 2024, EURO 2024 and ISMP 2024. The authors profited from the feedback of delegates to these meetings, and hope that it may lead to future extensions of the present findings. {\color{black} Last but not least, the authors are grateful to the evaluation team (reviewer and editor) whose thoughtful and constructive comments significantly contributed to improvements over an earlier version of this paper.}

%\appendix
%\section{Example Appendix Section}
%\label{app1}


%%       \cite{<label>} ==> [1]

\small

\bibliography{bibliography.bib}

\end{document}